\documentclass[twoside]{article}

\usepackage[accepted]{aistats2019}
\usepackage{amsmath,amsthm,epsfig,amssymb,amsbsy}
\usepackage{tikz}
\usepackage{pgfplots}
\usepackage{float}
\usepackage{caption}
\usepackage{subcaption}
\usepackage{enumerate}
\usepackage{comment}
\usepackage{algorithm}
\usepackage{algorithmic}
\usepackage{booktabs}       
\usepackage{amsfonts}       
\usepackage{nicefrac}       
\usepackage{relsize}

\usepackage{tikz}
\usetikzlibrary{external}
\tikzset{external/system call={lualatex
\tikzexternalcheckshellescape -halt-on-error -interaction=batchmode -jobname "\image" "\texsource"}}

\DeclareMathOperator*{\argmin}{arg\,min}

\DeclareMathOperator{\gph}{gph}
\DeclareMathOperator{\dom}{dom}

\DeclareMathOperator{\intr}{int}
\DeclareMathOperator{\bdry}{bdry}

\DeclareMathOperator{\dist}{dist}
\DeclareMathOperator{\lip}{lip}

\newtheorem{theorem}{Theorem}[section]
\newtheorem{corollary}[theorem]{Corollary}
\newtheorem{lemma}[theorem]{Lemma}

\newtheorem{definition}[theorem]{Definition}

\newtheorem*{theorem*}{Theorem}
\newtheorem*{lemma*}{Lemma}

\newcommand{\bR}{\mathbb{R}}

\newcommand{\bN}{\mathbb{N}}

\newcommand{\exR}{\overline{\mathbb{R}}}

\newcommand{\iprod}[2]{\left\langle#1,#2\right\rangle}

\newcommand{\prox}{P_\lambda^\phi}
\newcommand{\env}{e_{\lambda}^\phi}

\newcommand{\hPhi}{\widehat{\phi}}
\newcommand{\hhPhi}{\widehat{\widehat\phi \,}}

\newcommand{\iali}[1]{\begin{align}#1\end{align}}

\definecolor{mydarkblue}{rgb}{0,0.08,0.45}

\usepackage[colorlinks=true,
    linkcolor=mydarkblue,
    citecolor=mydarkblue,
    filecolor=mydarkblue,
    urlcolor=mydarkblue,
    pdfview=FitH]{hyperref}

\begin{document}

%
\runningtitle{Optimization of Inf-Convolution Regularized Nonconvex Composite Problems}

%
\runningauthor{Laude, Wu, Cremers}

\twocolumn[

\aistatstitle{Optimization of Inf-Convolution Regularized\\ Nonconvex Composite Problems}

\aistatsauthor{ Emanuel Laude \And  Tao Wu \And Daniel Cremers }
\aistatsaddress{Department of Informatics, Technical University of Munich, Germany} 
]

\begin{abstract}
In this work, we consider nonconvex composite problems that involve inf-convolution with a \emph{Legendre} function, which gives rise to an anisotropic generalization of the proximal mapping and Moreau-envelope. In a convex setting such problems can be solved via alternating minimization of a splitting formulation, where the consensus constraint is penalized with a Legendre function. In contrast, for nonconvex models it is in general unclear that this approach yields stationary points to the infimal convolution problem. To this end we analytically investigate local regularity properties of the Moreau-envelope function under prox-regularity, which allows us to establish the equivalence between stationary points of the splitting model and the original inf-convolution model.
We apply our theory to characterize stationary points of the penalty objective, which is minimized by the \emph{elastic averaging SGD} (EASGD) method for distributed training. Numerically, we demonstrate the practical relevance of the proposed approach on the important task of distributed training of deep neural networks.
\end{abstract}

\section{Introduction}
In this work, we are interested in optimizing nonconvex composite models which involve \emph{infimal convolutions} with \emph{Legendre} functions:
\begin{equation}\label{eq:composite_inf_conv}
\begin{aligned}
\underset{u\in \bR^n}{\text{minimize}}&& \env f(Au) + g(u).
\end{aligned}
\end{equation}
Here both $f:\bR^m \to \exR:=\bR\cup\{+\infty,-\infty\}$ and $g:\bR^n \to \exR$ are extended real-valued, proper\footnote{A function $f:\bR^m \to \exR$ is called proper if $f(\bar z)<\infty$ for some $\bar z\in \bR^m$ and $f(z)>-\infty$ for all $z\in \bR^m$.}
and lower semi-continuous (lsc) functions which are possibly nonconvex and nonsmooth, and $A \in \bR^{m \times n}$ is a coupling matrix. Let $\dom f:=\{z\in \bR^m: f(z) <\infty\}$ denote the domain of $f$. By $\env f$ we denote the infimal convolution of a function $f$ with some Legendre function $\phi: \bR^m \to \exR$; see Definition \ref{def:legendre}.
The infimal convolution of $f$ with a potential $\phi$ and scaling parameter $\lambda$ is defined as 
\begin{align} \label{eq:inf_conv}
\env f(v)=\inf_{z \in \bR^m} f(z) + \frac{1}{\lambda} \phi(v-z).
\end{align}
We shall refer to $\env f$ as the $\phi$-envelope of $f$, and the corresponding $\argmin$ map
\begin{align} \label{eq:phi_prox}
\prox f(v)=\argmin_{z \in \bR^m} f(z) + \frac{1}{\lambda} \phi(v-z),
\end{align}
as the $\phi$-proximal mapping of $f$ at $v$. 

Note that for $\phi=\frac{1}{2} \|\cdot\|^2$ they specialize to the classical Moreau-envelope and proximal mapping \cite{moreau1962fonctions,moreau1965proximite}. 

Under suitable assumptions that guarantee that the $\inf$ is attained when finite, $\env f$ yields a regularized variant of $f$ in the sense that the epigraph of $\env f$ is obtained via the Minkowski sum of the epigraphs of the individual functions $f$ and $\phi$ \cite[Exercise 1.28]{RoWe98}. 
For convex proper lsc $f$ and Lipschitz differentiable\footnote{A function is called Lipschitz differentiable if it is differentiable and its gradient is Lipschitz continuous.} $\phi$, $\env f$ is a Lipschitz differentiable approximation to $f$.
In contrast, when $f$ is nonconvex and nonsmooth, $\env f$ remains nonsmooth and nonconvex in general which renders the optimization of \eqref{eq:composite_inf_conv} challenging.

Inf-convolution models are well grounded in machine learning and signal processing. A variety of convex and nonconvex loss functions and regularizers can be written as an infimal convolution. There, the potential $\phi$ is chosen in accordance with the underlying noise prior, e.g., quadratic for Gaussian.

In addition, benefits of inf-convolution regularization (with quadratic $\phi$) have been observed empirically in neural network training in recent works \cite{ZCL15,chaudhari2018deep}. 

\subsection{Motivation}
To compute a stationary point of $\eqref{eq:composite_inf_conv}$, we resort to a splitting model
\begin{equation} \label{eq:composite_splitting}
\begin{aligned}
\underset{\substack{u\in \bR^n, ~z \in \bR^m}}{\text{minimize}}~
F(u, z):= f(z) + \frac{1}{\lambda}\phi(Au - z) + g(u), 
\end{aligned}
\end{equation}
where the violation of the constraint $Au = z$ is penalized with $\phi$. From Equation~\eqref{eq:inf_conv} it can be seen that model~\eqref{eq:composite_splitting} is equivalent to model~\eqref{eq:composite_inf_conv} in terms of global optima but in general not in terms of stationary points\footnote{A stationary point $(\bar{u}, \bar{z})$ of \eqref{eq:composite_splitting} is a point that satisfies the necessary first order optimality condition $0\in \partial F(\bar{u}, \bar{z})$, cf.~Fermat's rule generalized \cite[Theorem 10.1]{RoWe98}. 
When $(\bar{u}, \bar{z}) \in \dom F$ is feasible and $\phi$ is continuously differentiable on $\dom \phi$ open, $0\in \partial F(\bar{u}, \bar{z})$ is implied by \eqref{eq:cp_composite_splitting_1}--\eqref{eq:cp_composite_splitting_2} via \cite[Exercise 8.8 (c) and Proposition 10.5]{RoWe98}.
}.
The splitting formulation~\eqref{eq:composite_splitting} is amenable to alternating optimization, which, under mild assumptions, converges subsequentially to a stationarity point $(\bar{u}, \bar{z})$ of $F$, satisfying the following conditions (assuming $\dom \phi$ is open): $A\bar{u}-\bar{z} \in \dom \phi$, $\bar{u} \in \dom g$, $\bar{z} \in \dom f$, and
\begin{subequations}
\begin{align}
0 &\in \partial g(\bar{u}) + \tfrac{1}{\lambda} A^\top \nabla \phi(A\bar{u}-\bar{z}), \label{eq:cp_composite_splitting_1} \\
0 &\in \partial f(\bar{z}) - \tfrac{1}{\lambda}\nabla \phi(A\bar{u}-\bar{z}). \label{eq:cp_composite_splitting_2}
\end{align}
\end{subequations}
Here $\partial$ denotes the (limiting) subdifferential of a function, cf.~Definition \ref{def:general_subdiff}.

Meanwhile, in order for $\bar{u}$ to qualify as
a stationary point of the original problem \eqref{eq:composite_inf_conv} it must  satisfy
\begin{align}
0 \in\partial (\env f \circ A + g)(\bar{u}).
\label{eq:cp_composite_inf_conv}
\end{align}
When $\env f$ is smooth around $A \bar{u}$ and $\bar{u}\in \dom g$, the stationarity condition \eqref{eq:cp_composite_inf_conv} simplifies to 
\begin{align}
0 \in A^\top \nabla \env f(A\bar{u}) + \partial g(\bar{u}),
\label{eq:cp_composite_inf_conv_simp}
\end{align} 
via \cite[Exercise 8.8 (c)]{RoWe98}.

It is important to realize that conditions \eqref{eq:cp_composite_splitting_1}--\eqref{eq:cp_composite_splitting_2} do \emph{not} imply \eqref{eq:cp_composite_inf_conv} or \eqref{eq:cp_composite_inf_conv_simp} in general when $f$ is nonconvex.
This stands in stark contrast to the convex setting, where the stationarity condition~\eqref{eq:cp_composite_inf_conv_simp} (for quadratic $\phi=\frac{1}{2}\|\cdot\|^2$) can be guaranteed via the well known gradient formula for the Moreau-envelope \cite[Theorem 2.26]{RoWe98}: 
\begin{align} \label{eq:gradient_moreau_env_convex}
\nabla e_\lambda^{\|\cdot\|^2/2} f(v) = \frac{1}{\lambda} (v - P_\lambda^{\|\cdot\|^2/2} f (v)).
\end{align}
To this end, note that \eqref{eq:cp_composite_splitting_2} resembles the necessary (and in the convex setting also sufficient) optimality condition of the $\phi$-proximal mapping $\bar{z}=\prox (A\bar{u})$.
For this reason, a main focus of this work is to derive sufficient conditions that guarantee the \emph{translation of stationarity} in the more general nonconvex setting for (nonquadratic) $\phi$, which ultimately boils down to the following implication:
\begin{align} \label{eq:impl_cp_composite}
\eqref{eq:cp_composite_splitting_2} ~\Rightarrow ~ \tfrac1\lambda \nabla \phi(A\bar{u}-\bar{z}) =\nabla \env f (A \bar{u}).
\end{align}
In this sense, the implication in \eqref{eq:impl_cp_composite} is a generalization of the gradient formula \eqref{eq:gradient_moreau_env_convex} for the the $\phi$-envelope. 

\subsection{Contributions}
Our contributions are summarized as follows:
\begin{itemize}
\item We consider an anisotropic generalization of the proximal mapping and Moreau-envelope \cite{lescarret1967applications,wexler1973prox,combettes2013moreau} induced by a Legendre function in the nonconvex setting. More precisely we establish local regularity properties of the envelope function and proximal mapping under prox-regularity, including a generalization of the well known gradient formula for the Moreau-envelope. The translation of stationarity is a consequence of this theory.
\item We apply our theory to characterize stationary points of the model that is minimized by the elastic averaging SGD (EASGD) \cite{ZCL15} method for distributed training with anisotropic (i.e., non-quadratic) penalty functions. There, our theory can be invoked to obtain a robust measure of stationarity via the gradient of the Moreau-envelope. 
\item Numerically, we apply our algorithm to distributed training of deep neural networks and showcase merits of anisotropic inf-convolution potentials over standard quadratic in this context.
\end{itemize}

\section{Related Work}
Proximal mappings and Moreau-envelopes date back to the seminal papers of Moreau \cite{moreau1962fonctions,moreau1965proximite}.

\cite{lescarret1967applications,wexler1973prox,combettes2013moreau} consider an anisotropic generalization of the proximal mapping which is obtained by replacing the quadratic penalty with a Legendre function and study its properties in a convex setting. \cite{combettes2013moreau} relates the anisotropic proximal mapping to the Bregman proximal mapping (introduced in \cite{censor1992proximal,teboulle1992entropic} and investigated in \cite{bauschke2003bregman}) via a generalization of Moreau's decomposition \cite{moreau1962fonctions,moreau1965proximite}, which holds for convex functions. The Bregman prox and the anisotropic prox are different generalizations of the classical prox with complementary properties.

In the convex setting the Moreau-envelope has strong regularity properties such as Lipschitz differentiability. In the nonconvex setting the Moreau-envelope is nonsmooth in general. In \cite{PoRo96,RoWe98} the concept of prox-regular functions is introduced, which allows the authors to (locally) recover some of the properties known from the convex setting: These include the (local) single-valuedness of the prox and (local) Lipschitz differentiability of the envelope.

In \cite{lewis2009local,Och18}, prox-regularity is used to establish a local convergence result for alternating and averaged projections methods.
Similar to this work but specialized to quadratic $\phi$, in \cite{LWC18}, prox-regularity is utilized to show a translation of stationarity for the model \eqref{eq:composite_inf_conv} which is computationally resolved by multiblock ADMM.

More recently, in \cite{davis2019stochastic} the gradient of the Moreau-envelope has been used as a stationarity measure in stochastic optimization methods.

Inf-convolution regularization has recently been utilized in neural network training \cite{ZCL15,chaudhari2018deep}:
In \cite{ZCL15} the authors have considered the consensus training of deep neural networks by optimizing a relaxed consensus model of the form \eqref{eq:composite_splitting} with quadratic $\phi$. Similar algorithms were later connected to partial differential equations~\cite{chaudhari2018deep}. 

\section{Anisotropic Proximal Mapping}
In this section, we introduce the notion of Legendre functions that gives rise to an anisotropic generalization of the proximal mapping and Moreau-envelope investigated in the convex setting by \cite{lescarret1967applications,wexler1973prox,combettes2013moreau}. We establish a sufficient condition for the well-definedness of the anisotropic prox (denoted as $\phi$-prox) and envelope (denoted as $\phi$-envelope) in the nonconvex setting, based on a generalized notion of prox-boundedness \cite[Definition 1.23]{RoWe98}.

A proper convex lsc Legendre function is defined below according to \cite[Section 26]{Roc70}. Here $\partial \phi$ reduces to the classical convex subdifferential.
\begin{definition}[Legendre function] \label{def:legendre}
The proper convex lsc function $\phi:\bR^m \to \exR$ is
\begin{enumerate}
\item[\rm (i)] \emph{essentially smooth}, if the interior of the domain of $\phi$ is nonempty, i.e.~$\intr\dom \phi\neq \emptyset$, and $\phi$ is differentiable on $\intr\dom \phi$ such that $\|\nabla\phi(w^\nu)\|\to \infty$ whenever $w^\nu \to w \in \bdry \dom \phi$;
\item[\rm (ii)] \emph{essentially strictly convex}, if $\phi$ is strictly convex on every convex subset of $\dom \partial \phi:=\{w \in \bR^m : \partial \phi(w) \neq \emptyset \}$;
\item[\rm (iii)] \emph{Legendre}, if $\phi$ is both essentially smooth and essentially strictly convex.
\end{enumerate}
\end{definition}
Let $\phi^*$ denote the convex conjugate of $\phi$. Then, Legendre functions have the following essential properties:
\begin{lemma} \label{lem:legendre_props}
Let $\phi:\bR^m \to \exR$ be proper lsc convex and Legendre. Then $\phi$ has the following properties:
\begin{enumerate}
\item[\rm (i)] $\dom \partial \phi = \intr \dom\phi$, \cite[Theorem 26.1]{Roc70}. 
\item[\rm (ii)] $\nabla \phi: \intr \dom\phi \to \intr \dom\phi^*$ is bijective with inverse $\nabla \phi^*: \intr \dom\phi^* \to \intr \dom\phi$ with both $\nabla \phi$ and $\nabla \phi^*$ continuous on $\intr \dom\phi$ resp.~$\intr\dom\phi^*$, \cite[Theorem 26.5]{Roc70}.
\end{enumerate}
\end{lemma}

Overall we will make the following (additional) assumptions on $\phi$:
\begin{enumerate}
\item[\rm (A1)] $\phi:\bR^m \to \exR$ is proper lsc convex and Legendre.
\item[\rm (A2)]  $\dom \phi$ is open.
\item[\rm (A3)]  $\phi$ is twice continuously differentiable on $\intr\dom \phi$ with positive definite Hessian, i.e., $\nabla^2 \phi(w) \succ 0$ for any $w \in \intr\dom \phi$.
\item[\rm (A4)]  $\phi$ is \emph{super-coercive}, i.e., $\|\phi(w)\|/\|w\|\to \infty$ whenever $\|w\|\to\infty$.
\item[\rm (A5)]  $\phi(0) = 0$ and $\nabla \phi(0) = 0$.
\end{enumerate}
(A2) ensures that $\phi(w^\nu) \to \infty$, whenever $w^\nu \to w \in \bdry \dom \phi$. In particular this allows us to utilize alternating gradient descent steps as updates in our algorithm, cf.~Section~\ref{sec:optimization}. (A3) implies that $\phi$ is ``locally'' strongly convex and Lipschitz differentiable in the sense of \cite[Proposition 2.10]{bauschke2000dykstras}: For any compact and convex $K \subset \dom \partial \phi$, there are constants $\mu,\gamma>0$ such that for any $w_1, w_2 \in K$:
\begin{align*}
&\phi(w_1) \geq \phi(w_2) + \langle \nabla \phi(w_2), w_1 -w_2 \rangle + \frac{\mu}{2}\|w_1-w_2\|^2, \\
&\|\nabla \phi(w_1) - \nabla \phi(w_2) \| \leq \gamma \|w_1-w_2\|.
\end{align*}
Such functions are known under the term \emph{very strictly convex} \cite[Definition 2.8]{bauschke2000dykstras} which lie ``strictly between the class of strongly convex and the class of strictly convex functions'', \cite[Remark 2.9]{bauschke2000dykstras}.
(A4) is required later on in Section~\ref{sec:translation_stat} to show the translation of stationarity.
(A5) is technically not required in our theory. However, it naturally leads to a smoothing which under prox-regularity preserves stationarity (see Corollary~\ref{cor:stat_point_regularized_prox_reg}).

Examples for such $\phi$ include the scaled quadratic $\phi(w)=w^\top Q w$ (with matrix $Q$ symmetric positive definite) or a log-barrier function $\phi(w)= -\log(1- \|w\|^2)$. Further examples are provided in Section~\ref{sec:numerics}.

It is important to realize that for nonconvex $f$ the well-definedness of the $\phi$-proximal mapping and envelope requires additional assumptions: More precisely, we shall guarantee that $\env f$ is proper and $\prox f(v) \neq \emptyset$ for any $v \in \dom \env f$.
In addition, this condition allows us to extract a continuity property for the $\phi$-proximal mapping and envelope which is extensively needed later on to prove the desired translation of stationarity in Section~\ref{sec:translation_stat}:
\begin{definition}[$\phi$-prox-boundedness]
We say $f:\bR^m \to \exR$ is $\phi$-prox-bounded if there exists $\lambda > 0$ such that for any $\bar{v} \in \bR^m$ there exists $\epsilon > 0$ and a constant $\beta > -\infty$ such that
\begin{align} \label{eq:prox_bound_uniform}
\env f(v) \geq \beta
\end{align}
for any $v$ with $\|v-\bar{v}\|\leq \epsilon$. The supremum of the set of all such $\lambda$ is the threshold $\lambda_f$ of the $\phi$-prox-boundedness.
\end{definition}
When $f$ is bounded from below it is $\phi$-prox-bounded with threshold $\lambda_f = \infty$. 
Notably, in the classical case (when $\phi$ is quadratic) the definition can be made minimalistic, cf.~\cite[Definition 1.23]{RoWe98}: It suffices to assume the existence of some $\bar{v} \in \bR^m$ so that $\env f(\bar{v}) > -\infty$.

Overall we summarize below the properties of $\phi$-prox and envelope under $\phi$-prox-boundedness that shall be used along our course in the next section.
\begin{lemma} \label{lem:continuity}
Let $f:\bR^m \to \exR$ be proper lsc and $\phi$-prox-bounded with threshold $\lambda_f>0$. Then for any $\lambda \in (0,\lambda_f)$, $\prox f$ and $\env f$ have the following properties:
\begin{enumerate}
\item[\rm (i)] $\prox f(v) \neq \emptyset$ is compact for all $v \in \dom \env f=\dom f + \dom \phi$, whereas $\prox f(v) =\emptyset$ for $v \notin \dom \env f$.
\item[\rm (ii)] The $\phi$-envelope $\env f$ is continuous relative to $\dom \env f$.
\item[\rm (iii)] For any sequence $v^\nu \to \bar{v}$ contained in $\dom \env f$ and $z^\nu \in \prox f(v^\nu)$ we have $\{z^\nu\}_{\nu \in \bN}$ is bounded and all its cluster points $\bar{z}$ lie in $\prox f(\bar{v})$.
\end{enumerate}
\end{lemma}

\section{Translation of Stationarity} \label{sec:translation_stat}
In this section we prove the translation of stationarity, see \eqref{eq:impl_cp_composite}, under prox-regularity: We emphasize that it is a major concern of the splitting approach in the nonconvex setting to justify whether solving the splitting model \eqref{eq:composite_splitting} guarantees solving the original model \eqref{eq:composite_inf_conv}, both in terms of stationarity. 

To this end we recall the definitions of the regular and the limiting subdifferential according to \cite[Definition 8.3]{RoWe98}.
\begin{definition}[subdifferential] \label{def:general_subdiff}
Let $f:\bR^{m} \to \exR$ and $\bar{z}\in\dom f$ be given. For $y\in \bR^{m}$, we say 
\begin{enumerate}
\item[\rm (i)] $y$ is a regular subgradient of f at $\bar{z}$, written $y \in \widehat{\partial} f(\bar{z})$, if
\begin{align*}
\liminf_{\substack{z \to \bar{z} \\ z \neq \bar{z}}} \frac{f(z)-f(\bar{z})- \iprod{y}{z-\bar{z}}}{\|z-\bar{z}\|} \geq 0.
\end{align*}
We refer to the set $\widehat{\partial} f(\bar{z})$ as the regular subdifferential of $f$ at $\bar{z}$.
\item[\rm (ii)] $y$ is a (limiting) subgradient of $f$ at $y$, written $y \in \partial f(\bar{z})$, if there exist $z^\nu \to \bar{z}$ with $f(z^\nu) \to f(\bar{z})$ and $y^\nu \in \widehat{\partial} f(z^\nu)$ with $y^\nu \to y$.
We refer to the set $\partial f(\bar{z})$ as the (limiting) subdifferential of $f$ at $\bar{z}$.
\end{enumerate}
\end{definition}
We remark that for $f$ convex, both the regular and the limiting subdifferential coincide with the classical convex subdifferential, \cite[Proposition 8.12]{RoWe98}.

Next, we define prox-regularity of functions, according to \cite[Definition 13.27]{RoWe98}:
\begin{definition}[prox-regularity of functions]
Assume $f:\bR^m\to\exR$ is lsc and finite at $\bar{z}\in\bR^m$. We say $f$ is prox-regular at $\bar{z}$ for $\bar{y}\in\partial f(\bar{z})$ if there exist $\epsilon>0$ and $r\geq 0$ such that for all $\|z'-\bar{z}\|<\epsilon$
\begin{align}  \label{eq:prox_regularity}
f(z') \geq f(z)+\iprod{y}{z'-z}-\frac{r}{2}\|z'-z\|^2,
\end{align}
whenever $\|z-\bar{z}\|<\epsilon$, $f(z)-f(\bar{z})<\epsilon$, $y\in\partial f(z)$, $\|y-\bar{y}\|<\epsilon$. When this holds for all $\bar{y}\in\partial f(\bar{z})$, f is said to be prox-regular at $\bar{z}$.
\end{definition}
Prox-regularity is a local property in nature. Examples for (everywhere) prox-regular functions include: (i) proper, (weakly) convex, lsc functions; (ii) $\mathcal{C}^2$-functions; and (iii) indicator functions of $\mathcal{C}^2$-manifolds \cite{PoRo96,RoWe98}. For further examples we refer to \cite{PoRo96,RoWe98}.

Based on prox-regularity and $\phi$-prox-boundedness, we now extend \cite[Proposition 13.37]{RoWe98} to $\prox f$ and $\env f$ in our context (see Theorem \ref{thm:prox_regularity}) and eventually derive the translation of stationarity as desired (see Corollary \ref{cor:trstat}).
As a key ingredient in the proof of Theorem \ref{thm:prox_regularity} we may invoke the {generalized implicit function theorem} \cite[Theorem 2.1]{Rob80} \cite[Theorem 2B.5]{DoRo09} to assert $\prox f$ is locally a single-valued, Lipschitz map. 

\begin{theorem} \label{thm:prox_regularity}
Let $f:\bR^m\to\exR$ proper lsc and $\phi$-prox-bounded with threshold $\lambda_f$. Let $\bar{v}\in \bar{z}+\dom\phi$. Then for any $\lambda\in (0,\lambda_f)$ sufficiently small and $f$ finite and prox-regular at $\bar{z}$ for $\bar{y}\in\partial f(\bar{z})$ with 
\begin{align*}
\bar{y} =\frac1\lambda \nabla \phi(\bar{v}-\bar{z})
\end{align*}
the following statements hold true:
\begin{itemize}
\item[\rm (i)] $\prox f$ is a singled-valued, Lipschitz 
 map near $\bar{v}$ such that 
$\bar{z}=\prox f(\bar{v})$ and 
\begin{align}
\prox f(v) 
&= (I + \nabla \phi^* \circ \lambda T)^{-1}(v),
\end{align}
where $T$ is the $f$-attentive $\epsilon$-localization of $\partial f$ near $(\bar{z},\bar{y})$, i.e. the set-valued mapping $T:\bR^m \rightrightarrows \bR^m$ defined by $T(z):=\{y \in \partial f(z) : \|y -\bar{y}\|< \epsilon\}$
if $\|z - \bar{z}\| < \epsilon$ and $f(z)<f(\bar{z})+\epsilon$, and $T(z):=\emptyset$ otherwise.
\item[\rm (ii)] $\env f$ is Lipschitz differentiable around $\bar{v}$ with 
\begin{equation} \label{eq:gradient_formula}
\nabla \env f(v)=\frac1\lambda\nabla \phi(v-z).
\end{equation}
\end{itemize}
\end{theorem}
\begin{proof}
We provide a short proof sketch for part (i). For a detailed proof of part (i) and part (ii) we refer to the supplements.

(i) Using the definition of prox-regularity, the assumptions in the theorem and the continuity property of the prox from Lemma~\ref{lem:continuity} (which holds under $\phi$-prox-boundedness) it can be proven that for some $\lambda \in (0, \lambda_f)$ sufficiently small for any $\xi$ sufficiently near $0$ we have
$
\xi \in T(z)-\frac1\lambda\nabla \phi(\bar{v}-z),
$
for some $z$ near $\bar{z}$. Furthermore, using the definition of prox-regularity and Assumption (A3), it can be shown that $T-\frac1\lambda \nabla \phi(\bar{v}-\cdot)$ is strongly monotone.

This implies that $\xi\mapsto\left(T-\frac1\lambda\nabla \phi(\bar{v}-\cdot)\right)^{-1}(\xi)$ is a single-valued, Lipschitz map in a neighborhood of $0$ such that $\left(T-\frac1\lambda\nabla \phi(\bar{v}-\cdot)\right)^{-1}(0)=\bar{z}$. 
Invoking the generalized implicit function theorem 
\cite[Theorem 2B.7]{DoRo09}, we assert that $v\mapsto\prox f(v)=\left(T-\frac1\lambda\nabla \phi(v-\cdot)\right)^{-1}(0)$ is a single-valued, Lipschitz map in a neighborhood of $\bar{v}$ 
such that $\bar{z}=\prox f(\bar{v})$.
\end{proof}

As an immediate consequence of the above theorem, the implication in \eqref{eq:impl_cp_composite} holds true and the translation of stationarity is attained under prox-regularity.

\begin{corollary}[translation of stationarity] \label{cor:trstat}
Let $(\bar{u},\bar{z})$ be a stationary point for the splitting model \eqref{eq:composite_splitting} satisfying $A\bar{u}-\bar{z} \in \dom \phi$, $\bar{u} \in \dom g$, $\bar{z} \in \dom f$ and conditions \eqref{eq:cp_composite_splitting_1}--\eqref{eq:cp_composite_splitting_2}. Let $f:\bR^m\to\exR$ be $\phi$-prox-bounded with threshold $\lambda_f$ and prox-regular at $\bar{z}$ for $\bar{y}:=\frac{1}{\lambda}\nabla \phi(A\bar{u}-\bar{z})$. Then, for $\lambda \in (0, \lambda_f)$ sufficiently small, the stationarity condition \eqref{eq:cp_composite_inf_conv_simp} is fulfilled, i.e., $\bar{u}$ is a stationary point for the inf-convolution model \eqref{eq:composite_inf_conv}.
\end{corollary}
\begin{proof}
Invoking Theorem \ref{thm:prox_regularity} with $(\bar{z},\frac{1}{\lambda}\nabla \phi(A\bar{u}-\bar{z}))$ and $\lambda \in (0, \lambda_f)$ sufficiently small, we obtain $A^\top \nabla \env f(A\bar{u})=\frac1\lambda A^\top \nabla \phi(A\bar{u}-\bar{z})$. In combination with \eqref{eq:cp_composite_splitting_1}, this yields 
$
0 \in A^\top \nabla \env f(A\bar{u}) + \partial g(\bar{u}).
$
Since $\env f$ is continuously differentiable around $A \bar{u}$ and $\bar{u} \in \dom g$, this implies \eqref{eq:cp_composite_inf_conv_simp} due to \cite[Exercise 8.8 (c)]{RoWe98}.
\end{proof}

\section{Application to Distributed Training} \label{sec:optimization}
In this section, we present a stochastic alternating minimization scheme (Algorithm \ref{alg:salm}) tailored to distributed empirical risk minimization in \eqref{eq:model_easgd}.
As an interesting special case, we consider scenarios where all workers have access to the entire training set and the method specializes to elastic averaging SGD (EASGD) \cite{ZCL15}. 
\subsection{Inexact Alternating Minimization}
For the optimization of the splitting problem \eqref{eq:composite_splitting} one typically resorts to alternating minimization where the variables $u$ and $z$ are updated as:
\begin{subequations}
\begin{align}
u^{t+1} \in \displaystyle \argmin_{u\in\bR^n} ~g(u)+ \frac{1}{\lambda}\phi(Au-z^t), \\
z^{t+1} \in \displaystyle \argmin_{z\in \bR^m} ~ \frac{1}{\lambda}\phi(Au^{t+1}-z) + f(z).
\end{align}
\end{subequations}
Such a scheme is known as the Gauss-Seidel or block coordinate descent method and has been investigated in a general setting by, e.g., \cite{auslender1976optimisation,bertsekas1989parallel,tseng2001convergence}.

In our case we regard alternating minimization as a generic scheme where the subproblems (in particular the $z$-update) may be solved approximately (e.g., by replacing the function with a surrogate) as long as convergence to a stationary point of the splitting problem can be guaranteed. Then, the translation of stationarity from Corollary~\ref{cor:trstat} applies under prox-regularity.

For instance, the vanilla Gauss-Seidel method can be extended with a proximal regularization as in \cite{attouch2010proximal}. When $\phi$ is Lipschitz differentiable, the coupling term $\phi(Au-z)$ can be replaced by a proximal linearization in $z$ (resp.~$u$) at $z^t$ (resp.~$u^t$), so that the updates become proximal gradient steps on $F$ as in proximal alternating linearized minimization (PALM) \cite{bolte2014proximal}.

Alternating minimization can also incorporate stochastic gradient updates, as described in the next subsection.

\subsection{Stochastic Alternating Minimization for Distributed Training} \label{sec:samdt}
In distributed learning, a set of $M$ workers collaborates on the training of a model parameterized by consensus weights $u$. To this end, with access to a prescribed subset (indexed by $\mathcal{I}_j$) of the full training set (indexed by $\mathcal{I}$), each individual worker trains its local copy $z_j$ of the model parameters $u$ under a (relaxed) consensus constraint $u=z_j$. As a consequence, all workers can update their copies in parallel.

In terms of the model \eqref{eq:composite_splitting}, this is formulated as follows. For the remainder of this section let $f:\bR^{nM} \to \bR$ with 
$$
f(z)=\sum_{j=1}^M f_j(z_j)
$$
be a separable sum of (regularized) continuously differentiable empirical risks $f_j:\bR^{n} \to \bR$, each assigned to worker $j$.
More specifically each $f_j$ is defined as
$$
f_j(z_j)=  \frac1{|\mathcal{I}_j|} \sum_{i\in \mathcal{I}_j} \ell(h_i, H(x_i; z_j)) + R(z_j),
$$
for training pairs $(x_i, h_i)_{i \in \mathcal{I}_j}$, some prediction function $H(\cdot;z)$ parameterized by weights $z$, a loss function $\ell(\cdot,\cdot)$ and a regularizer $R$. Furthermore, let $g:=0$ and $A:=[I, \cdots, I]^\top \in \bR^{(nM) \times n}$, and set the Legendre penalty as
$$
\phi(w):= \sum_{j=1}^M \hPhi(w_j)
$$
which is separable over the copies $w_j$. Overall model~\eqref{eq:composite_splitting} then reads:
\begin{align} \label{eq:model_easgd}
\min_{u, (z_j)_{j=1}^M \in \bR^{n}} F(u,z)=\sum_{j=1}^M f_j(z_j) + \frac1\lambda\hPhi(u - z_j),
\end{align}
where $\hPhi(u - z_j)$ loosely enforces the consensus constraint $u = z_j$.

The stochastic instance of the alternating minimization scheme is formulated in Algorithm~\ref{alg:salm} below. Notably, Algorithm~\ref{alg:salm} specializes to EASGD \cite{ZCL15} in the isotropic case, i.e. when $\phi$ is quadratic and all workers have access to the entire training set $\mathcal{I}$, i.e. $\mathcal{I}_j= \mathcal{I}$.
\setlength{\textfloatsep}{10pt}
\begin{algorithm}[h!]
\caption{Stochastic Alternating Linearized Minimization}
\label{alg:salm}
\begin{algorithmic}[1]
\FORALL{$t=1,2,\ldots$}
\STATE Choose proper step sizes $\sigma_t, \tau > 0$.
\STATE $u^{t+1} = u^t - \tau A^\top \nabla \phi( Au^t - z^{t})$.
\STATE Draw a sample of the random variable $\xi^t$.
\STATE Compute $\Delta(z^t; \xi^t)$ as an unbiased estimate of the gradient of $f+\frac1\lambda \phi(Au^{t+1} - \cdot)$ at $z^t$.
\STATE $z^{t+1} = z^t - \sigma_t \Delta(z^t; \xi^t)$.
\ENDFOR
\end{algorithmic}
\end{algorithm}

To obtain an unbiased estimate of the gradient $\nabla_z F(u^{t+1}, z^{t})$ of $F(u^{t+1}, \cdot)= f + \frac1\lambda \phi(A u^{t+1} - \cdot)$ at $z^t$, worker $j$ samples a uniformly random minibatch $\mathcal{B}_j^t$ from $\mathcal{I}_j$, and computes the standard stochastic gradient, 
\begin{align}
\delta_j^t =\frac{1}{|\mathcal{B}_j^t|}\sum_{i\in \mathcal{B}_j^t} \nabla (\ell(h_i, H(x_i; \cdot)))(z_j^t). 
\end{align}
Then we may define $\Delta(z^t; \xi^t):=(\Delta_j(z_j^t; \xi_j^t))_{j=1}^M$ for
\begin{align}
\Delta_j(z_j^t; \xi_j^t) = \delta_j^t + \nabla R(z_j^t) - \frac1\lambda \nabla \phi(u^t - z_j^t).
\end{align}

The $u$-update in Algorithm~\ref{alg:salm} combines the current model parameters $z_j^t$ of the workers into a consensus model $u^{t+1}$. Notably, for the isotropic case and the particular choice $\tau=1/M$, the consensus update reduces to the arithmetic mean of the copies~$z_j^t$.

As a main difference to EASGD, for general $\phi$ the $u$-update can be regarded as a more general form of averaging. For instance, for the non-admissible choice $\phi=\|\cdot\|_1$, it is well known that minimization of $F$ w.r.t.~$u$ yields the (componentwise) median.

A convergence proof for the stochastic method is beyond the scope of the paper. Instead we focus on the characterization of the stationary points $(\bar{u}, \bar{z})$, that the algorithm attempts to find.

Invoking the translation of stationarity reveals that the solution $\bar u$ is stationary w.r.t.~the sum of $\phi$-envelopes $\sum_{j=1}^M e_\lambda^{\hPhi}f_j$.
If furthermore all workers can sample from the entire training set $\mathcal{I}$, i.e., $\mathcal{I}_j= \mathcal{I}$ and therefore all $f_j = \hat{f}$ are equal, perfect consensus $\bar u = \bar z_j$ holds at stationary points for some finite penalty parameter $1/\lambda>0$, which translates to a stationary point of the unregularized problem $\min_u \hat{f}(u)$ satisfying $0\in \partial \hat{f}(\bar u)$.
This is not trivial as the workers may follow different paths to different stationary points if not coupled tightly (due to stochasticity).
In addition, our theory shows that the $\phi$-envelope is Lipschitz differentiable at $\bar u$ and $\nabla \env \hat f(\bar u) = 0$, implying that whenever $u^\nu \to \bar u$ it holds that $\|\nabla \env \hat f(u^\nu)\| \to 0$. The gradient norm of the $\phi$-envelope may thus serve as a measure of stationarity, more robust compared to $\dist(0,\partial \hat f(u^\nu))\to 0$, see~\cite{davis2019stochastic}.
All of these properties are formally stated in the following corollary. 

\begin{corollary} \label{cor:stat_point_regularized_prox_reg}
Let $(\bar{u},\bar{z})$ be a stationary point for the splitting model \eqref{eq:model_easgd} satisfying $\bar{u}-\bar{z}_j \in \dom \hPhi$ and conditions \eqref{eq:cp_composite_splitting_1}--\eqref{eq:cp_composite_splitting_2}. If $f:\bR^m\to\exR$ is prox-regular at $\bar{z}$ for $\bar y = \frac1\lambda \nabla \phi(A\bar u - \bar z)$ and $\phi$-prox-bounded with threshold $\lambda_f$, then for $\lambda \in (0, \lambda_f)$ sufficiently small, $\bar u$ is a stationary point of $\sum_{j=1}^M e_\lambda^{\hPhi}f_j$, i.e., $\sum_{j=1}^M e_\lambda^{\hPhi}f_j$ is Lipschitz differentiable around $\bar u$ and
$$\sum_{j=1}^M \nabla e_\lambda^{\hPhi}f_j(\bar u) = 0.$$
If in addition $\mathcal{I}_j = \mathcal{I}$ for all $1\leq j \leq M$ and therefore $f_j = \hat{f}$ the stationarity condition reduces to
$$
0 = \nabla e_\lambda^{\hPhi} \hat{f}(\bar{u}),
$$
\begin{figure*}[h!]
\centering
\begin{subfigure}[b]{0.4\linewidth}
        \includegraphics[width=\textwidth]{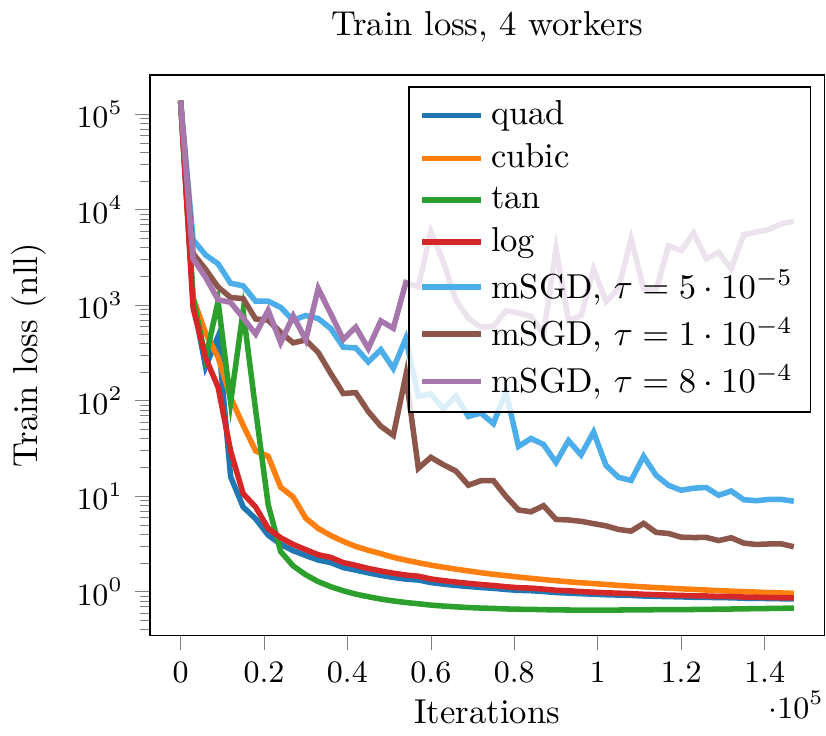}	
\end{subfigure}
\begin{subfigure}[b]{0.4\linewidth}
	\includegraphics[width=\textwidth]{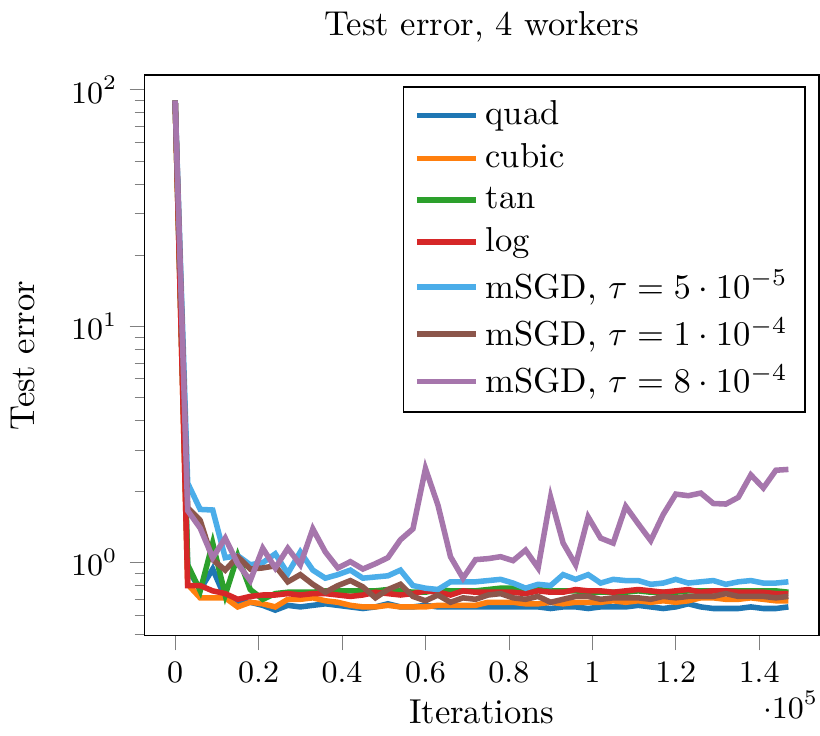} 
\end{subfigure}\\
\begin{subfigure}[b]{0.4\linewidth}
        \includegraphics[width=\textwidth]{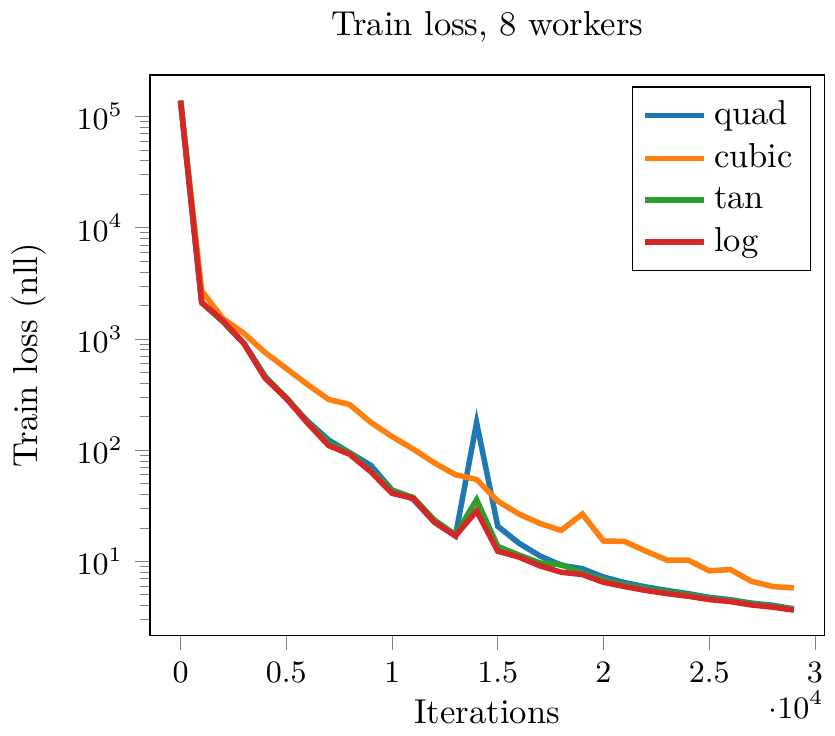}	
\end{subfigure}
\begin{subfigure}[b]{0.4\linewidth}
	\includegraphics[width=\textwidth]{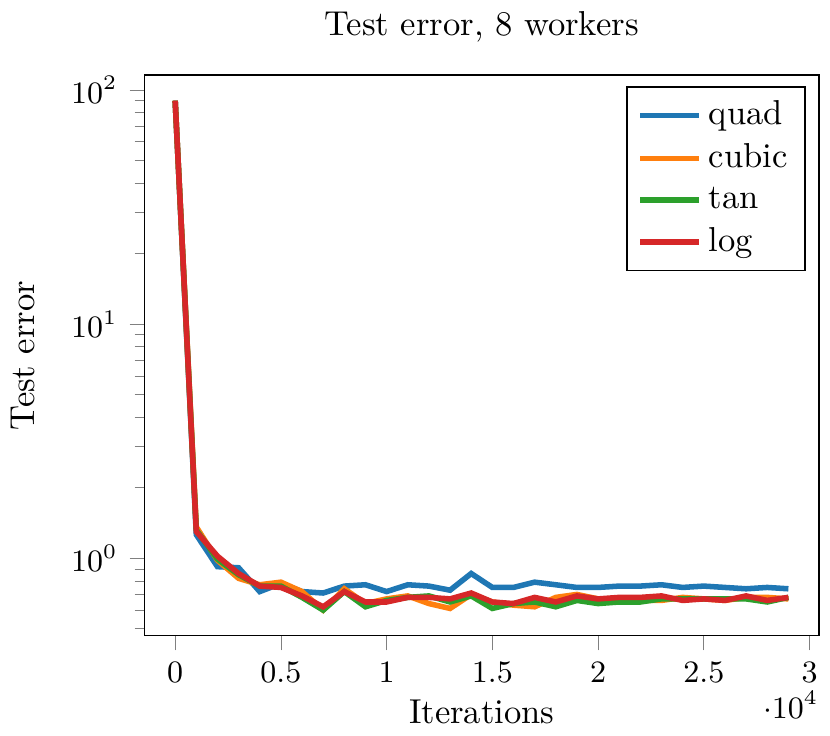}
\end{subfigure}
\caption{Convergence plots for stochastic distributed training with Algorithm~\ref{alg:salm} and classical Nesterov momentum SGD on MNIST for 4 workers (upper row) and 8 workers (lower row). In the 4 workers setting each worker completes 50 epochs. In the 8 worker setting each worker completes 10 epochs. In both cases our method is run with a learning rate of $\tau=0.005$, much higher than the highest one possible with SGD (for $\tau=8\cdot10^{-4}$ mSGD already becomes unstable). For all algorithms the batch-size is 20.}
\label{fig:mnist}
\end{figure*}
and it holds that,
$\bar{u}=\bar{z}_j$ for all $j$ and $0 \in \partial \hat{f}(\bar{u})$.
\end{corollary}
\begin{proof}
The first part of the corollary follows directly from Corollary~\ref{cor:trstat} and the special structure of $A,f,g$ and $\phi$.
For the second part we invoke Theorem~\ref{thm:prox_regularity} (i) and obtain that for some $\lambda > 0$, sufficiently small $P_\lambda^{\hPhi} \hat{f}$ is single-valued at $\bar{u}$ and that
$
\bar{z}_j = P_\lambda^{\hPhi} \hat{f}(\bar{u}),
$
showing that all $\bar{z}_j$'s are equal.
From \eqref{eq:cp_composite_splitting_1} it follows that
$$
0 = A^\top \nabla \phi(A \bar{u} - \bar{z})  = M \nabla\hPhi(\bar{u} - P_\lambda^{\hPhi} \hat{f}(\bar{u})).
$$
By Assumptions (A1), (A5) and Lemma~\ref{lem:legendre_props} we have $\nabla \hPhi(w) = 0 $ if and only if $w = 0$, and therefore $\bar{u} = \bar{z}_j$ and $\bar{y}_j = 0$. From \eqref{eq:cp_composite_splitting_2} we know $0 \in \partial \hat{f}(\bar{z}_j) = \partial \hat{f}(\bar{u})$.
\end{proof}

\section{Numerical Experiments} \label{sec:numerics}
\begin{table*}[h]
\caption{
Comparison of different potentials on the stochastic consensus training of a deep neural network on MNIST. Results after $30,000$ iterations, so that each of the 4 workers completes 10 epochs with batch size $20$. For each potential we report the best values in performance over all configurations. Our evaluation suggests that quadratic is not the best smoothing potential in general.
}
\center
{\footnotesize
\begin{tabular}{cccccc}
\cmidrule{1-6}
$\phi$  & Objective & Train Loss (nll) & Train Error & Test Loss (nll) & Test Error \\
\midrule
quad, \cite{ZCL15} & 2.49 & 2.47 & 0.00 \% & 294.12 & 0.65 \% \\[0.5ex]
cubic & 7.44 & 6.87 & 0.00 \% &  {\bf{231.42}} &  {\bf{0.56}} \% \\[0.5ex]
tan & {\bf{0.90}} & {\bf{0.83}} & 0.00 \% & 300.84 & 0.62 \% \\[0.5ex]
tan-sep & 0.91 & 0.87 & 0.00 \% & 306.91 & 0.64 \% \\[0.5ex]
log & 2.36 & 2.35 & 0.00 \% & 299.60 & 0.64 \% \\[0.5ex]
log-sep & 2.46 & 2.44 & 0.00 \% & 299.82 & 0.67 \% \\[0.5ex]
\bottomrule
\end{tabular}
}
\label{tab:mnist}
\end{table*}
In the experiments we consider the stochastic distributed training of a deep neural network with Algorithm~\ref{alg:salm} where all workers have access to the entire training set, i.e. $\mathcal{I}_j= \mathcal{I}$ for all $j$ in terms of the model~\eqref{eq:model_easgd}.
More precisely, we report comparisons of different choices of potentials.
We manually set $\hPhi$ as
\begin{align}
\hPhi(w_j):=\sum_{l=1}^L \hhPhi\left(\frac{w_{jl}}{\eta}\right),
\end{align}
where $\eta$ is a scaling parameter (in addition to $\lambda$), and $l$ indexes the (learnable) layers.
The different choices of $\hhPhi$ are summarized in Table~\ref{tab:potentials}.
\begin{table}[h]
\caption{
Choices for $\hhPhi$. To ensure that cubic satisfies Assumption (A3) a small quadratic may be added.}
\center

{\footnotesize
\begin{tabular}{ccc}
\cmidrule{1-3}
 acronyms& $\hhPhi$ & $\dom \hhPhi$ \\
\midrule
quad & $\|\cdot\|^2$ & $\bR^m$ \\[0.5ex]
cubic & $\|\cdot\|_3^3$ & $\bR^m$ \\[0.5ex]
tan & $\tan(\|\cdot\|^2)$ & $\left\{ w : \|w\| < \sqrt\frac{\pi}2 \right\}$ \\[0.5ex]
tan-sep & $\sum_i \tan((\cdot)_i^2)$ & $\left\{ w : |w_i| < \sqrt\frac{\pi}2\right\}$ \\[0.5ex]
log & $-\log(1-\|\cdot\|^2)$ & $\left\{ w : \|w\| < 1 \right\}$ \\[0.5ex]
log-sep & $\sum_i-\log(1-(\cdot)_i^2)$ & $\left\{ w : |w_i| < 1\right\}$ \\[0.5ex]
\bottomrule
\end{tabular}
}
\label{tab:potentials}
\end{table}
Since both the log- and tan-potentials have bounded domains, we incorporate a line search in our algorithm to ensure that the iterates stay feasible. Note that cubic does not satisfy Assumption (A3) as its Hessian is 0 at 0. Yet we include it in our numerical evaluation.

We apply a variant of our Algorithm \ref{alg:salm} with constant step size, that in addition incorporates Nesterov momentum \cite{sutskever2013importance} in the SGD updates of the $z_j$. Note that for $\phi := \frac{1}{2}\|\cdot\|^2$ this algorithm specializes to the synchronous EASGD, resp.~mEASGD \cite{ZCL15}. 

We report results of the training of a classifier on the MNIST dataset resorting to the standard LeNet-5 CNN architecture \cite{lecun1998gradient} given as
\begin{quotation}
\noindent
Conv$_{20, 5, 1}$ $\to$ ReLU $\to$ Pool$_{2,2}$ $\to$ Conv$_{50, 5, 1}$ $\to$ ReLU $\to$ Pool$_{2,2}$  $\to$ FC $\to$ Softmax
\end{quotation}

In Table \ref{tab:mnist} we compare different potentials $\hhPhi$. To this end we perform a grid search over different learning rates $\sigma=\tau \in \{0.001, 0.005\}$, momentum parameters $\kappa \in \{0.9, 0.99\}$, and different scalings of the potentials $\lambda \in \{0.1, 0.05, 0.025, 0.01, 0.005, 0.0025\}$, $\eta \in \{0.5, 1, 2\}$. We set the number of workers to 4, the batch size to 20 and the regularization parameter $\nu=10^{-4}$. We run the algorithm for $30,000$ iterations, so that each worker completes 10 epochs.
For each potential we report the best performances over all scalings and configurations in Table \ref{tab:mnist}.

In terms of training loss, the tan- and log-potentials seem slightly superior, while the non-separable variants yield slightly better performance than the separable ones. The cubic potential performs worst in terms of training loss.
All potentials consistently yield 0\% training error after 10 epochs (for each worker).

Notably, we observe that the cubic potential performs better, in terms of low test error, than all other potentials. This is even true for a whole range of scalings. 

In Figure~\ref{fig:mnist}, 
we show convergence plots for a representative configuration of our Algorithm~\ref{alg:salm} for 4 and 8 workers respectively. In addition, we show a comparison to plain SGD with Nesterov momentum (mSGD). The different EASGD variants ``see'' 4 to 8 times more training examples than mSGD within one iteration. As a result the distributed method is stable at much higher learning rates and requires fewer iterations than mSGD to achieve low objective value.

\section{Conclusion}
In this work we have considered an anisotropic generalization of the proximal mapping and Moreau-envelope. We derived a gradient formula for the Moreau-envelope in the nonconvex setting based on prox-regularity. This allows us to equivalently (in terms of stationary points) reformulate the problem as a splitting problem which is amenable to (stochastic) alternating minimization. As an application of our theory we characterize stationary points of the penalty objective that is optimized by the elastic averaging SGD method. There, our theory can be used to obtain a robust measure of stationarity. Through numerical validations we demonstrated the relevance of our theory and algorithm on the important task of consensus training of deep neural networks.

\subsection*{Acknowledgement}
The work was supported by the DFG Research Grant ``Splitting Methods for 3D Reconstruction and SLAM". 
\newpage

\bibliographystyle{plain}
\bibliography{submission}

\begin{thebibliography}{10}

\bibitem{attouch2010proximal}
H.~Attouch, J.~Bolte, P.~Redont, and A.~Soubeyran.
\newblock Proximal alternating minimization and projection methods for
  nonconvex problems: An approach based on the {Kurdyka}-{{\L}ojasiewicz}
  inequality.
\newblock {\em Mathematics of Operations Research}, 35:438--457, 2010.

\bibitem{auslender1976optimisation}
A.~Auslender.
\newblock {\em Optimisation: {M}{\'e}thodes {N}um{\'e}riques}.
\newblock Masson, 1976.

\bibitem{bauschke2003bregman}
H.~H. Bauschke, J.~M. Borwein, and P.~L. Combettes.
\newblock {Bregman} monotone optimization algorithms.
\newblock {\em {SIAM} Journal on Control and Optimization}, 42(2):596--636,
  2003.

\bibitem{bauschke2000dykstras}
H.~H. Bauschke and A.~S. Lewis.
\newblock {Dykstras} algorithm with {Bregman} projections: A convergence proof.
\newblock {\em Optimization}, 48(4):409--427, 2000.

\bibitem{bertsekas1989parallel}
D.~P. Bertsekas and J.~N. Tsitsiklis.
\newblock {\em Parallel and Distributed Computation: Numerical Methods}.
\newblock Prentice-Hall, Inc., Upper Saddle River, NJ, USA, 1989.

\bibitem{bolte2014proximal}
J.~Bolte, S.~Sabach, and M.~Teboulle.
\newblock Proximal alternating linearized minimization for nonconvex and
  nonsmooth problems.
\newblock {\em Mathematical Programming}, 146(1-2):459--494, 2014.

\bibitem{censor1992proximal}
Y.~Censor and S.~A. Zenios.
\newblock Proximal minimization algorithm with {D}-functions.
\newblock {\em Journal of Optimization Theory and Applications},
  73(3):451--464, 1992.

\bibitem{chaudhari2018deep}
P.~Chaudhari, A.~Oberman, S.~Osher, S.~Soatto, and G.~Carlier.
\newblock Deep relaxation: partial differential equations for optimizing deep
  neural networks.
\newblock {\em Research in the Mathematical Sciences}, 5(3):30, 2018.

\bibitem{combettes2013moreau}
P.~L. Combettes and N.~N. Reyes.
\newblock Moreau's decomposition in {Banach} spaces.
\newblock {\em Mathematical Programming}, 139(1-2):103--114, 2013.

\bibitem{davis2019stochastic}
Damek Davis and Dmitriy Drusvyatskiy.
\newblock Stochastic model-based minimization of weakly convex functions.
\newblock {\em {SIAM} Journal on Optimization}, 29(1):207--239, 2019.

\bibitem{DoRo09}
A.~L. Dontchev and R.~T. Rockafellar.
\newblock {\em Implicit Functions and Solution Mappings: A View From
  Variational Analysis}.
\newblock Springer, New York, 2009.

\bibitem{LWC18}
E.~Laude, T.~Wu, and D.~Cremers.
\newblock A nonconvex proximal splitting algorithm under {Moreau}-{Yosida}
  regularization.
\newblock In {\em Artificial Intelligence and Statistics (AISTATS)}, 2018.

\bibitem{lecun1998gradient}
Y.~LeCun, L.~Bottou, Y.~Bengio, and P.~Haffner.
\newblock Gradient-based learning applied to document recognition.
\newblock {\em Proceedings of the IEEE}, 86(11):2278--2324, 1998.

\bibitem{lescarret1967applications}
C.~Lescarret.
\newblock Applications ``prox'' dans un espace de {Banach}.
\newblock {\em Comptes Rendus de l'Acad{\'e}mie des Sciences de {Paris}
  S{\'e}rie A}, 265:676--678, 1967.

\bibitem{lewis2009local}
A.~S. Lewis, D.~R. Luke, and J.~Malick.
\newblock Local linear convergence for alternating and averaged nonconvex
  projections.
\newblock {\em Foundations of Computational Mathematics}, 9(4):485--513, 2009.

\bibitem{moreau1962fonctions}
J.-J. Moreau.
\newblock Fonctions convexes duales et points proximaux dans un espace
  {Hilbertien}.
\newblock {\em Comptes Rendus de l'Acad{\'e}mie des Sciences de Paris S{\'e}rie
  A}, 255:2897--2899, 1962.

\bibitem{moreau1965proximite}
J.-J. Moreau.
\newblock Proximit{\'e} et dualit{\'e} dans un espace hilbertien.
\newblock {\em Bulletin de la Soci{\'e}t{\'e} Math{\'e}matique de {France}},
  93(2):273--299, 1965.

\bibitem{Och18}
P.~Ochs.
\newblock Local convergence of the heavy-ball method and {iPiano} for
  non-convex optimization.
\newblock {\em Journal of Optimization Theory and Applications}, 177:153--180,
  2018.

\bibitem{PoRo96}
R.~A. Poliquin and R.~T. Rockafellar.
\newblock Prox-regular functions in variational analysis.
\newblock {\em Trans. Amer. Math. Soc.}, 348:1805--1838, 1996.

\bibitem{Rob80}
S.~M. Robinson.
\newblock Strongly regular generalized equations.
\newblock {\em Mathematics of Operations Research}, 5:43--62, 1980.

\bibitem{Roc70}
R.~T. Rockafellar.
\newblock {\em Convex Analysis}.
\newblock Princeton University Press, New Jersey, 1970.

\bibitem{RoWe98}
R.~T. Rockafellar and R.~J.-B. Wets.
\newblock {\em Variational Analysis}.
\newblock Springer, New York, 1998.

\bibitem{sutskever2013importance}
I.~Sutskever, J.~Martens, G.~Dahl, and G.~Hinton.
\newblock On the importance of initialization and momentum in deep learning.
\newblock In {\em International Conference on Machine Learning (ICML)}, pages
  1139--1147, 2013.

\bibitem{teboulle1992entropic}
M.~Teboulle.
\newblock Entropic proximal mappings with applications to nonlinear
  programming.
\newblock {\em Mathematics of Operations Research}, 17(3):670--690, 1992.

\bibitem{tseng2001convergence}
P.~Tseng.
\newblock Convergence of a block coordinate descent method for
  nondifferentiable minimization.
\newblock {\em Journal of Optimization Theory and Applications},
  109(3):475--494, 2001.

\bibitem{wexler1973prox}
D.~Wexler.
\newblock Prox-mappings associated with a pair of {Legendre} conjugate
  functions.
\newblock {\em Revue {fran{\c{c}}aise} d'automatique informatique recherche
  op{\'e}rationnelle}, 7(R2):39--65, 1973.

\bibitem{ZCL15}
S.~Zhang, A.~Choromanska, and Y.~LeCun.
\newblock Deep learning with elastic averaging {SGD}.
\newblock In {\em Advances in Neural Information Processing Systems (NIPS)},
  pages 685--693, 2015.

\end{thebibliography}

\end{document}